\documentclass[12pt]{amsart}
\usepackage[english]{babel}
\usepackage{graphicx}
\usepackage{psfrag}
\newtheorem{thm}{Theorem}[section]

\newtheorem{lem}[thm]{Lemma}
\newtheorem{prop}[thm]{Proposition}

\theoremstyle{definition}

\newtheorem{rem}[thm]{Remark}

\def \n{\noindent }
\def \bs{\bigskip}
\def \F{\mathcal{F}}
\def \A{\mathcal A}
\def \B{\mathcal B}
\def \D{\mathcal D}
\def \P{\mathcal P}
\def \M{\mathcal M}
\def \Q{\mathcal Q}
\def \Z{\mathbb Z}

\def \Z{\mathbb Z}
\def \av{{\rm av}}

\def \s{S}
\def \r{R}
\def \t{T}

\def \ty{\infty}
\title[A Note on Minimal zero-sum sequences over $\Z$]{A Note on Minimal zero-sum sequences over $\Z$}

\author[P. A. Sissokho]{Papa A. Sissokho}
\address{Mathematics Department\\ Illinois State University\\
Normal, Illinois 61790--4520, U.S.A.}
\email{psissok@ilstu.edu}
\begin{document}

\date{}

\begin{abstract}
A {\em zero-sum sequence over ${\mathbb Z}$} is a sequence with terms in ${\mathbb Z}$ that sum to $0$.
It is called {\em minimal} if it does not contain a proper zero-sum subsequence.
Consider a minimal zero-sum sequence over ${\mathbb Z}$ with positive terms 
$a_1,\ldots,a_h$ and negative terms $b_1,\ldots,b_k$.  We prove that 
$h\leq \lfloor \sigma^+/k\rfloor$ and $k\leq \lfloor \sigma^+/h\rfloor$, where 
$\sigma^+=\sum_{i=1}^h a_i=-\sum_{j=1}^k b_j$. These bounds are tight and improve
upon previous results. We also show a natural partial order structure on the collection 
of all minimal zero-sum sequences over the set $\{i\in {\mathbb Z}:\; -n\leq i\leq n\}$
for any positive integer $n$.
\end{abstract}



\maketitle
%==========================================================================
\section{Introduction}\label{sec:intro}
%===============================================================================
We shall follow the notation and definitions in ‎Grynkiewicz's new monograph, and
refer the reader to it for the definitions that were omitted here.

For all integers $x$ and $y$ with $x\leq y$, let $[x,y]=\{i\in\Z:\; x\leq i\leq y\}$.
Let $G_0$ a non-empty subset of an additive abelian group $G$.  
Let  $\F(G_0)$ denote the free multiplicative abelian monoid with basis $G_0$, and whose
elements are the (unordered) sequences with terms in $G_0$. The identity element of $\F(G_0)$, also 
called {\em trivial sequence}, is the sequence with no terms. The operation of $\F(G_0)$  is 
the {\em sequence concatenation} product that takes $\r,\t\in\F(G_0)$ 
to $\s=\r\cdot \t\in\F(G_0)$. In this case, we say that $\r$ (respectively, $\t$) is a {\em subsequence} of $\s$. 
For every $\s=s_1\cdot\ldots\cdot s_t\in\F(G_0)$, let 
\begin{align}\label{def:not}
&\mbox{{\em the length} of $\s$, denoted by $|\s|$, be }|\s|=k;\cr 
%&\mbox{{\em the multiplicity} of $s_i$ in $\s$, denoted by $\v_{s_i}(S)$, be } \v_{s_i}(S)=|\{j\in[1,t]:\; s_j=s_i\}|;\cr
&\mbox{{\em the sum} of $\s$, denoted by $\sigma(\s)$, be }\sigma(\s)=s_1+s_2+\ldots+s_t;\\
&\mbox{{\em the average} of $\s$, denoted by $\s_\av$, be }\s_\av=\sigma(\s)/|\s|;\cr
&\mbox{{\em the infinite norm} of $\s$, denoted by $\|\s\|_\ty$, be }\|\s\|_\ty=\sup\limits_{1\leq i\leq t}|s_i|.\notag
\end{align}
For any $g\in G$ and any integer $d\geq 0$, we let
\[g^{[d]}=\underbrace{g\cdot\ldots\cdot g}_d,\]
where $g^{[d]}$ denotes the empty sequence if $d=0$.  

A {\em zero-sum sequence over $G_0$} is a sequence $\s\in\F(G_0)$ such that $\sigma(\s)=0$. 
Such a sequence is called {\em minimal} if it does not contain a proper 
non-trivial zero-sum subsequence. Then, the submonoid
\[\B_0=\B(G_0)=\{\s \in\F(G_0):\;\sigma(\s)=0\}\]
of $\F(G_0)$ is a Krull monoid (e.g., see~\cite{Gr}).
The set $\A(\B_0)$ of the {\em atoms} of $\B_0$ is the set of all minimal zero-sum sequences in $\B_0$. 
A characterization of $\A(\B_0)$ would shed some light on the factorization properties of $\B_0$ 
(e.g., see~\cite{GH,GGG}). 

Given a minimal zero-sum sequence $\s=s_1\cdot\ldots\cdot s_t\in \A(\B_0)$, we are interested in bounding 
its length in function of its terms $s_i$ for $i\in[1,t]$. We are also interested in finding a natural 
structure for $\A(\B_0)$ when $G_0$ (and thus, $\B_0$) is finite.

The study of zero-sum sequences in $\B(G)$, when $G$ a finite cyclic group,
is a very active area of research (e.g., see~\cite{Ad,C,EGZ,Ga,LP,Po,XS}) with applications to  
Factorization Theory (e.g., see~\cite{Ba,GG,Ge,GH}). Similar, but less extensive, investigations have been 
carried out when $G$ is an infinite cyclic group (e.g., see~\cite{BCRSS,CSS,GGG,GGSS}). 

For all $\s\in\B(\Z)$ with $|\s|$ finite and $|\s|>1$, there exist positive integers 
$a_1,\ldots,a_n$ and $b_1,\ldots,b_m$ with $a_1\leq \ldots\leq a_n$ and $b_1\leq\ldots\leq b_m$, 
such that 
\begin{equation}\label{eq:stdf}
\s^+=\prod_{i=1}^na_i^{[x_i]},\; \s^-=\prod_{j=1}^m(-b_j)^{[y_j]},\;\mbox{ and } \s=\s^+\cdot\s^-,
\end{equation}
where $x_i$ and $y_j$ are positive integers for all $i\in[1,n]$ and $j\in[1,m]$.

In his work on Diophantine linear equations, Lambert~\cite{La} 
proved the following theorem.
\begin{thm}[Lambert~\cite{La}]\label{thm:La}
Let $\s$ be a minimal zero-sum sequence over $\Z$ with $|\s|$ finite and $|\s|>1$. 
If $\s$ is as in~\eqref{eq:stdf}, then 
\[|\s^+|\leq \|\s^-\|_\ty=b_m \mbox{ and } |\s^-|\leq \|\s^+\|_\ty=a_n.\]
\end{thm}
This was reformulated and reproved in the language of sequences by Baginski et al.~\cite{BCRSS}. 
Perhaps due to inconsistent notation across various areas, Theorem~\ref{thm:La} has been independently
rediscovered by Diaconis et al.~\cite{DGS}, and Sahs et al.~\cite{SST}.
Currently, the best bounds for $|\s^+|$ and $|\s^-|$ are due to Henk-Weismantel~\cite{HW}.
They proved the following theorem for which Theorem~\ref{thm:La} is a special case upon setting $\ell=m$ and $k=n$. 
\begin{thm}[Henk-Weismantel~\cite{HW}]\label{thm:HW}
Let $\s$ be a minimal zero-sum sequence over $\Z$ with $|\s|$ finite and $|\s|>1$. 
If $\s$ is as in~\eqref{eq:stdf}, then  

\n $(J_\ell):\;$  $\;|\s^+|\leq b_\ell -\sum_{j=1}^{\ell-1}\left\lfloor\frac{b_\ell-b_j}{a_n}\right\rfloor y_j+\sum_{j=\ell+1}^{m}\left\lceil\frac{b_j-b_\ell}{a_1}\right\rceil y_j$ for all $\ell\in[1,m]$,

\n $(I_k):\;$ $\;|\s^-|\leq a_k-\sum_{i=1}^{k-1}\left\lfloor \frac{a_k-a_i}{b_m}\right\rfloor x_i +\sum_{i=k+1}^{n}\left\lceil \frac{a_i-a_k}{b_1}\right\rceil x_i$ for all $k\in[1,n]$.
\end{thm}
In this paper, we improve on Theorem~\ref{thm:HW} by proving the following theorem.
\begin{thm}\label{thm:main}
Let $\s$ be a minimal zero-sum sequence over $\Z$ with $|\s|$ finite and $|\s|>1$. 
If $\s$ is as in~\eqref{eq:stdf}, then  
\[|\s^+|\leq \left\lfloor-\s^-_\av\right\rfloor= \left\lfloor\frac{\sum_{j=1}^{m}b_jy_j}{\sum_{j=1}^{m} y_j} \right\rfloor
 \mbox{ and } 
\;|\s^-|\leq \left\lfloor\s^+_\av \right\rfloor=\left\lfloor\frac{\sum_{i=1}^{n}a_ix_i}{\sum_{i=1}^{n}x_i} \right\rfloor.\]
\end{thm}
The bounds in theorems~\ref{thm:La}--\ref{thm:main} are all tight for the
minimal zero-sum sequences 
\[\s=a^{[\frac{b}{\gcd(a,b)}]}\cdot (-b)^{[\frac{a}{\gcd(a,b)}]},\]
for all positive integers $a$ and $b$.
On the other hand, if we consider the minimal zero-sum sequence $\s=3^{[1]}\cdot4^{[2]}\cdot(-1)^{[2]}\cdot(-9)^{[1]}$, 
then Theorem~\ref{thm:La} yields $|\s^+|\leq 9$ and $|\s^-|\leq 4$, Theorem~\ref{thm:HW} yields $|\s^+|\leq 4$ and 
$|\s^-|\leq 4$, while Theorem~\ref{thm:main} yields the tight bounds $|\s^+|\leq 3$ and $|\s^-|\leq 3$.

In Section~\ref{sec:main}, we prove Theorem~\ref{thm:main} by refining the method of
Sahs et. al~\cite{SST}. In Section~\ref{sec:appl}, we define a natural partial order on 
the set $\A(\B_0)$ of minimal zero-sum sequences and discuss its relevance. In Section~\ref{sec:app},
we show that the bounds in Theorem~\ref{thm:main} are always sharper or equivalent to the bounds 
in Theorem~\ref{thm:HW}.
%==========================================================================
\section{Proofs of Theorem~\ref{thm:main}}\label{sec:main}\
%==========================================================================
Let $G$ be an additive abelian group, and let $\s=s_1\cdot s_2 \ldots \cdot s_t\in\F(G)$. 
For all $i,j\in[1,t]$ such that $i\not=j$, 
let $\s'$ be the sequence obtained by removing the terms $s_i$ and $s_j$ from $\s$ and inserting 
(anywhere) the term 
$s_i+s_j$. We call this process an {\em $(s_i,s_j)$-derivation} and say that 
$\s'$ is {\em $(s_i,s_j)$-derived} from $\s$. We also say that $\s'$ is {\em derived} from $\s$
without specifying the pair $(s_i,s_j)$.
For instance, if $\s=2^{[3]}\cdot(-3)^{[2]}$, then $\s'=2^{[2]}\cdot(-3)\cdot(-1)$ is $(2,-3)$-derived 
from $\s$, and $\s'=4^{[1]}\cdot2^{[1]}\cdot(-3)^{[2]}$ is $(2,2)$-derived from $\s$.

We will use the following lemma, which is a special case of Lemma~2 in Sahs et. al~\cite{SST}.
For the sake of completeness, we include a very short proof of it here.
\begin{lem}\label{lem:1}\
Let $G$ be an additive abelian group.
Let $\s=s_1\cdot s_2 \ldots \cdot s_t$ be a minimal zero-sum sequence over $G$, and let $i,j\in[1,t]$ 
be such that $i\not=j$. 
If $\s'$ is $(s_i,s_j)$-derived from $\s$, then $\s'$ is also a minimal zero-sum sequence over $G$. 
\end{lem}
\begin{proof}
By definition $\s'$ is a zero-sum sequence over $G$ since $s_i+s_j\in G$ and 
\[\sigma(\s')=\sigma(s)-s_i-s_j+(s_i+s_j)=\sigma(\s)=0.\]
Suppose that $\s'$ is not minimal. Then there exist nontrivial zero-sum subsequences $\r$ and $\t$
such that $\s'=\r\cdot \t$, and the specific term $s_i+s_j$ (there may be other copies of
$s_i+s_j$ in $\s'$ and $\s$) is a subsequence of either $\r$ or $\t$, and not both.  
Thus, either $\r$ or $\t$ is a proper zero-sum subsequence of $\s$. This would contradict 
the minimality of $\s$. Thus, $\s'$ is minimal zero-sum sequence.   
\end{proof}
We now prove our main theorem.
\begin{proof}[Proof of Theorem~\ref{thm:main}]\

Let $\s$ be a minimal zero-sum sequence over $\Z$ with $|\s|$ finite and $|\s|>1$. Then, 
there exist positive integers $a_1,\ldots,a_n$ and $b_1,\ldots,b_m$ with $a_1\leq \ldots\leq a_n$ and 
$b_1\leq\ldots\leq b_m$, such that  
\begin{equation*}
\s^+=\prod_{i=1}^n a_i^{[x_i]},\; \s^-=\prod_{j=1}^m (-b_j)^{[y_j]},\;\mbox{ and } \s=\s^+\cdot\s^-,
\end{equation*}
where $x_i$ and $y_j$ are positive integers for all $i\in[1,n]$ and $j\in[1,m]$.

We shall prove by induction on $|\s|\geq 2$ that 
\begin{equation}\label{eq:ind}
|\s^+|\leq -\s^-_\av\quad \mbox{ and }\quad  |\s^-|\leq \s^+_\av.
\end{equation}
If $|\s|=2$, then we must have $m=n=1$, $\s=a_1\cdot(-b_1)$, and $a_1-b_1=0$. 
Since $a_1>0$ and $b_1>0$, the statement~\eqref{eq:ind} clearly holds.
Assume that $|\s|\geq 2$ and that~\eqref{eq:ind} holds for all minimal zero-sum 
sequence $\r$ such that $2\leq |\r|<|\s|$.

If $a_i=b_j$ for some $i\in[1,n]$ and $j\in[1,m]$, then we must have $\s=a_i\cdot(-b_j)$. 
Otherwise, $\s'=a_i\cdot(-b_j)$ would be a proper zero-sum subsequence of $\s$, which
would contradict the minimality of $\s$.
Thus, we may assume that 
\[\{a_1,\ldots,a_n\}\cap\{b_1,\ldots,b_m\}=\emptyset.\] 
Without loss of generality, we may also assume that $a_n=\|\s^+\|_\ty>\|\s^-\|_\ty=b_m$. 

To prove the inductive step, we first show that $|\s^+|\leq -\s^-_\av$.
Since $x_n>0$, $y_m>0$, and $a_n-b_m>0$, we can use Lemma~\ref{lem:1} to perform an 
$(a_n,-b_m)$-derivation from $\s$, and obtain the minimal zero-sum sequence
\[\r=(a_n-b_m)^{[1]}\cdot a_n^{[x_n-1]}\cdot\prod_{i=1}^{n-1}a_i^{[x_i]}\cdot
(-b_m)^{[y_m-1]}\prod_{j=1}^m (-b_j)^{[y_j]},\]
where we omit the term $a_n$ if $x_n=1$ and the term $(-b_m)$ if $y_m=1$. 

Since $|\r|=|\s|-1$, it follows from the induction hypothesis that 
\begin{equation}\label{eq:c1}
|\r^+|=1+(x_n-1)+\sum_{i=1}^{n-1}x_i=\sum_{i=1}^n x_i\leq -\r^-_\av=
\frac{(y_m-1)b_m+\sum_{j=1}^{m-1} y_jb_j}{(y_m-1)+\sum_{j=1}^{m-1} y_j}.
\end{equation}

Since $b_m=\|\s^-\|_\ty\geq \|\r^-\|_\ty$, it follows from~\eqref{eq:c1} that
\[|\r^+|=\sum_{i=1}^n x_i\leq 
\frac{b_m+(y_m-1)b_m+\sum_{j=1}^{m-1} y_jb_j}{1+(y_m-1)+\sum_{j=1}^{m-1} y_j}
=\frac{-\sigma(\s^-)}{|\s^-|}=-\s^-_\av.\]
Thus, 
\begin{equation}\label{eq:c2}
\s^+=\sum_{i=1}^n x_i=|\r^+|\leq -\s^-_\av.
\end{equation}

\bs
Next, we show that $|\s^-|\leq \s^+_\av$.
Since $\sigma(\s)=0$, it follows that $\sigma(\s^+)=-\sigma(\s^-)$. 
This observation and~\eqref{eq:c2} yield 
\begin{equation}\label{eq:c3}
|\s^+|\leq -\s^-_\av =\frac{-\sigma(\s^-)}{|\s^-|}=\frac{\sigma(\s^+)}{|\s^-|}
\Longrightarrow |\s^-|\leq \frac{\sigma(\s^+)}{|\s^+|}=\s^+_\av.
\end{equation}
Since $|\s^+|$ and $|\s^-|$ are integers, the theorem follows from~\eqref{eq:c2} 
and~\eqref{eq:c3} by taking the floors of $\s^+_\av$ and $-\s^-_\av$. 

\end{proof}
\begin{rem}
Let $\s$ is as in~\eqref{eq:stdf} and suppose that there exists $t\in[1,m]$ such that 
\begin{equation}\label{eq:rem1}
a_n>b_t>-\s^-_\av=\frac{\sum_{j=1}^{m}b_jy_j}{\sum_{j=1}^{m} y_j}.
\end{equation}
Then, the $(a_n,-b_t)$-derivation on $\s$ yields the minimal zero-sum sequence
\[\r=(a_n-b_t)^{[1]}\cdot a_n^{[x_n-1]}\cdot\prod_{i=1}^{n-1}a_i^{[x_i]}\cdot
(-b_t)^{[y_t-1]}\prod_{j=1,\;j\not=t}^m (-b_j)^{[y_j]}.\]
Thus, by applying Theorem~\ref{thm:main} to $\r$, we obtain
\begin{equation}\label{eq:rem2}
|\s^+|=\sum_{i=1}^n x_i=|\r^+|\leq \left\lfloor-\r^-_\av\right\rfloor.
\end{equation}
Since $-\r^-_\av < -\s^-_\av$ (by the definition of $\r$ and~\eqref{eq:rem1}),
the bound for $|\s^+|$ in~\eqref{eq:rem2} is sometimes better than the bound 
$|\s^+|\leq \left\lfloor-\s^-_\av\right\rfloor$ given by Theorem~\ref{thm:main}.
By symmetry, we may sometimes obtain a better bound for $|\s^-|$ in a similar manner.
\end{rem}
%==========================================================================
\section{The structure of the minimal zero-sum sequences}\label{sec:appl}\
%==========================================================================
Let $G_0$ be a finite subset of $\Z$. 
We are interested in finding a natural structure on the set $\A(\B_0)$ of minimal zero-sum 
sequences in $\B_0=\B(G_0)$.
As mentioned in the introduction, $\A(\B_0)$ is also the set of atoms of the Krull monoid 
$\B_0$. There are other interesting interpretations of $\A(\B_0)$.
In the context of Diophantine linear equations (e.g., see~\cite{HW,La,Pot}), $\A(\B_0)$ 
correspond to the union of all {\em Hilbert bases}\footnote{This union is also known as the 
{\em Graver basis} of the corresponding {\em toric ideal} (e.g., see~\cite{StTh}).}, 
which are minimal generating sets of all the solutions.
In the context integer partitions, each sequence $\s=a_1\cdot\ldots\cdot a_p\cdot(-b_1)\cdot\ldots\cdot(-b_q)\in\A(\B_0)$ 
such that $p+q\geq 3$, $a_i>0$ for $i\in[1,p]$, and $b_j>0$ for $j\in[1,q]$, corresponds to the 
{\em primitive partition identity} $a_1+\ldots+a_p=b_1+\ldots+b_q$ (see~\cite[p.~1]{DGS}). 
Primitive partition identities were studied by Diaconis et al.~\cite{DGS} who were motivated applications in
 Gr\"obner bases, computational statistics, and integer programming (e.g., see~\cite{St,StTh}). 

In the process of characterizing $\A(\B_0)$, we assume that $\s=s_1\cdot\ldots\cdot s_t\in\A(\B_0)$ is equivalent to
$-\s=(-s_1)\ldots (-s_t)\in\A(\B_0)$ and we only include one of them in $\A(\B_0)$.
For any positive integer $n$, defined the {\em $n$-derived set}, $\D_n(\s)$, of $\s=s_1\cdot\ldots\cdot s_t\in\B(\Z)$ by
\begin{equation*}
\D_n(\s)=\left\{\s':\; i,j\in[1,t],\, i\not=j,\mbox{ $\s'$ is $(s_i,s_j)$-derived, and $\|\s'\|_\ty\leq n$}\right\}.
\end{equation*}
Given $\r,\s\in \B(\Z)$, we write  $\r\prec_n\s$ if and only if $\r=\s$ or $\r\in\D_n(\s)$.
% is derived from $\s$.

The following proposition is a direct consequence of Lemma~\ref{lem:1}. 
\begin{prop}\label{prop:1}
Let $n$ be a positive integer, $G_0=[-n,n]$, and $\B_0=\B(G_0)$.\\
\n $(i)$  If $\s\in \A(\B_0)$, then $\D_n(\s)\subseteq \A(\B_0)$.\\
\n $(ii)$ $\P_n=\left(\A(\B_0),\prec_n\right)$ is a poset.
\end{prop}
For instance, if $\s=2^{[3]}\cdot (-3)^{[2]}$, then Figure~\ref{fig1} shows the poset $\P_3$.
%\[\D_3(\s)=\left\{2^{[2]}\cdot(-1)\cdot (-3),2\cdot 1\cdot (-3),2\cdot (-1)^{[2]},2\cdot (-2),1\cdot (-1),0\right\}.\]
Note that $\s'=2^{[3]}\cdot (-6)$ is $(-3,-3)$-derived from $\s$, but $\s'\not\in\D_3(\s)$ since $\|\s'\|_\ty=6>3$.

\bs
\begin{figure}[ht!]
\psfrag{1}[][]{\tiny $2^{[3]}\cdot (-3)^{[2]}$}
\psfrag{2}[][]{\tiny $2^{[2]}\cdot(-3)\cdot (-1)$}
\psfrag{3}[][]{\tiny $2\cdot 1\cdot (-3)$}
\psfrag{4}[][]{\tiny $2\cdot (-2)$}
\psfrag{5}[][]{\tiny $0$}
\psfrag{6}[][]{\tiny $1^{[3]}\cdot (-3)$}
\psfrag{7}[][]{\tiny $\quad 1^{[2]}\cdot(-2)\simeq2\cdot(-1)^{[2]}$}
\psfrag{8}[][]{\tiny $3\cdot (-3)$}
\psfrag{9}[][]{\tiny $1\cdot (-1)$}
\centerline{\includegraphics[width=2.6in]{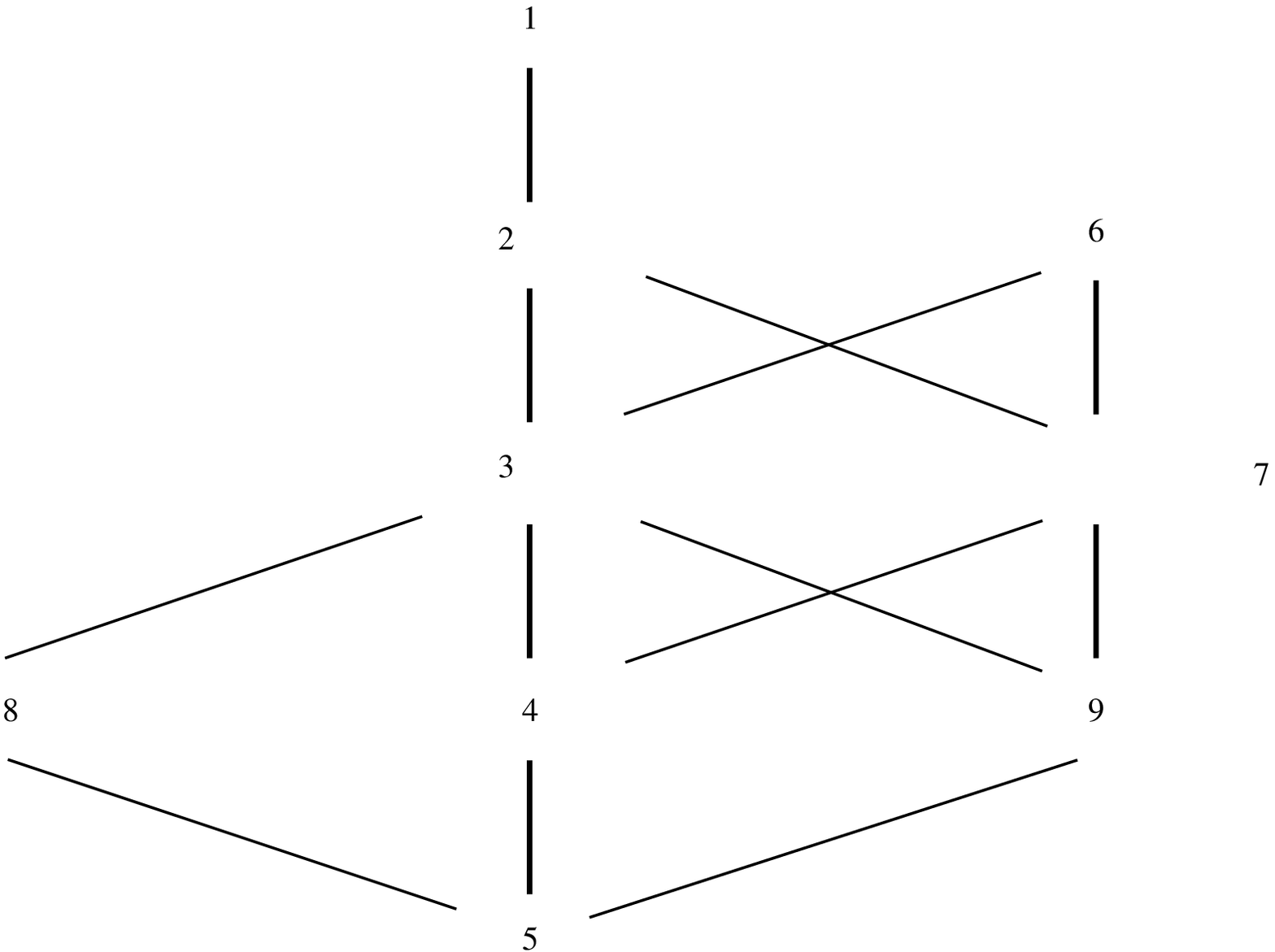}}
\caption{The poset $\P_3$}\label{fig1}
\end{figure}

Let $\M_n$ be the set of maximal elements of the poset $\P_n$ in Proposition~\ref{prop:1}, i.e., 
$\M_n$ contains all minimal sequences $\r\in\A(\B_0)$ that cannot be derived from any $\s\in\A(\B_0)$. 
Then the following proposition is immediate.
\begin{prop}\label{prop:2}
Let $n$ be a positive integer, $G_0=[-n,n]$, and $\B_0=\B(G_0)$.
If $\Q$ is a set such that $\M_n\subseteq \Q\subseteq \A(\B_0)$, then 
\[\A(\B_0)=\Q\cup\left(\bigcup\limits_{\s\in\Q}\D_n(\s)\right),\]
where we assume that $\s\in \A(\B_0)$ is equivalent to $-\s\in \A(\B_0)$.
\end{prop}
For instance, Figure~\ref{fig1} shows that
\[\M_3=\big\{2^{[3]}\cdot (-3)^{[2]},1^{[3]}\cdot (-3)^{[1]}\big\}.\] 
We also verified that 
\begin{equation}\label{eq:Mn}
\M_n\subseteq \left\{a^{[\frac{b}{\gcd(a,b)}]}\cdot (-b)^{[\frac{a}{\gcd(a,b)}]}:\; a,b\in[1,n]\right\}
\;\mbox{ for $n\in[1,5]$}.
\end{equation} 
However, by using the 4ti2--software package~\cite{4ti2}, we found that~\eqref{eq:Mn} does not 
hold for $n=6$. In particular,  
\begin{align*}
\M_6-\left\{a^{[\frac{b}{\gcd(a,b)}]}\cdot (-b)^{[\frac{a}{\gcd(a,b)}]}:\; a,b\in[1,6]\right\}
=\Big\{ & 2^{[2]}\cdot 3^{[1]}\cdot 5^{[1]}\cdot (-6)^{[2]},\\
		 & 1^{[1]}\cdot 3^{[1]}\cdot 4^{[2]}\cdot (-6)^{[2]}\Big\}.
\end{align*}

Determining $\M_n$ (or a small enough superset of $\M_n$), for all $n>0$, would directly yield an algorithm 
for generating $\P_n$, and an approach for computing the cardinality of $\A(\B_0)$ (e.g., by studying the 
M\"obius function of $\P_n$).
%==========================================================================
\section{Comparison of the bounds in Theorems~\ref{thm:HW}\&\ref{thm:main}}\label{sec:app}\
%==========================================================================
In this section, we show that the bounds in Theorem~\ref{thm:main} are in general 
sharper or equivalent to the bounds in Theorem~\ref{thm:HW}. 
To do this, we will show that it is enough to compare those two theorems 
for sequences $\s$ (where $\s$ is as in~\eqref{eq:stdf}) such that
\begin{equation}\label{eq:appi}
 a_1\leq |\s^-|=\sum_{j=1}^{m}y_j \leq a_n  \mbox{ and } b_1\leq |\s^+|=\sum_{i=1}^{n}x_i\leq b_m.
\end{equation}

First, note that it follows from Theorem~\ref{thm:La} that 
\begin{equation}\label{eq:app0}
\sum_{j=1}^{m}y_j=|\s^-| \leq a_n \mbox{ and } \sum_{i=1}^{n}x_i=|\s^+|\leq b_m.
\end{equation}
Let $\ell\in[1,m]$, $k\in[1,n]$, and consider the upper bounds
\begin{equation}\label{eq:ul} 
 U_{J_\ell}= b_\ell -\sum_{j=1}^{\ell-1}\left\lfloor\frac{b_\ell-b_j}{a_n}\right\rfloor y_j
 + \sum_{j=\ell+1}^{m}\left\lfloor\frac{b_\ell-b_j}{a_1}\right\rfloor y_j
\end{equation}
\begin{equation}\label{eq:uk} 
U_{I_k}=a_k -\sum_{i=1}^{k-1}\left\lfloor\frac{a_k-a_i}{b_m}\right\rfloor x_i
+\sum_{i=k+1}^{n}\left\lceil \frac{a_i-a_k}{b_1}\right\rceil x_i
\end{equation}
in the inequalities $(J_\ell)$  and $(I_k)$ in Theorem~\ref{thm:HW}, where $a_1\leq\ldots\leq a_n$ 
and $b_1\leq\ldots\leq b_m$. 

Without loss of generality, assume that $b_m\geq a_n$. Then 
$\left\lfloor\frac{a_k-a_i}{b_m}\right\rfloor=0$ for $1\leq i<k\leq n$, and it follows from~\eqref{eq:uk} that 
\begin{equation}\label{eq:app1.1} 
U_{I_k}\geq a_k\geq a_1\mbox{ for all $k\in[1,n]$}.
\end{equation}
Thus, it follows from~\eqref{eq:app0}, \eqref{eq:app1.1}, and the 
fact that $\s^+_\av\geq a_1$, that Theorem~\ref{thm:HW} and~\ref{thm:main} 
can only give meaningful upper bounds for $|\s^-|$ if
\begin{equation}\label{eq:app1}
a_1\leq |\s^-|=\sum_{j=1}^{m}y_j \leq a_n.
\end{equation}
Next, it follows from the definition of $-\s^-_\av$ in~\eqref{def:not} that
\begin{align}\label{eq:app2}
-\s^-_\av=\frac{-\sigma(\s^-)}{|\s^-|}
&= \frac{\sum_{j=1}^{m}b_jy_j}{\sum_{j=1}^{m}y_j}\cr
&= \frac{\sum_{j=1}^{m}b_\ell y_j-\sum_{j=1}^{\ell-1}(b_\ell-b_j)y_j+\sum_{j=\ell+1}^{m}(b_j-b_\ell)y_j}{\sum_{j=1}^{m}y_j}\cr
&= b_\ell-\frac{\sum_{j=1}^{\ell-1}(b_\ell-b_j)y_j}{\sum_{j=1}^{m}y_j}+\frac{\sum_{j=\ell+1}^{m}(b_j-b_\ell)y_j}{\sum_{j=1}^{m}y_j}.
\end{align}
Since $a_1\leq\ldots\leq a_n$ and $b_1\leq\ldots\leq b_m$, it follows from~\eqref{eq:app1} and~\eqref{eq:app2} that
\begin{align}\label{eq:app3}
-\s^-_\av 
&\leq  b_\ell-\sum_{j=1}^{\ell-1}\frac{(b_\ell-b_j)y_j}{a_n}+
\sum_{j=\ell+1}^{m}\frac{(b_j-b_\ell)y_j}{a_1}\cr
&\leq b_\ell -\sum_{j=1}^{\ell-1}\left\lfloor\frac{b_\ell-b_j}{a_n}\right\rfloor y_j+
\sum_{j=\ell+1}^{m}\left\lceil\frac{b_j-b_\ell}{a_1}\right\rceil y_j=U_{J_\ell}.
\end{align}
Thus, Theorem~\ref{thm:main} and~\eqref{eq:app3} yield
\begin{equation}\label{eq:fin1}
|\s^+|\leq \left\lfloor-\s^-_\av\right\rfloor\leq -\s^-_\av \leq U_{J_\ell}
%b_\ell -\sum_{j=1}^{\ell-1}\left\lfloor\frac{b_\ell-b_j}{a_n}\right\rfloor y_j+
%\sum_{j=\ell+1}^{m}\left\lceil\frac{b_j-b_\ell}{a_1}\right\rceil y_j,
\end{equation}
which implies inequality $(J_\ell)$ in Theorem~\ref{thm:HW}. 

Moreover, it follows from \eqref{eq:fin1} and the definition of $-\s^-_\av$ that 
\begin{equation}\label{eq:app4}
b_1\leq  -\s^-_\av \leq U_{J_\ell}.
\end{equation}
Thus, it follows from~\eqref{eq:app0} and \eqref{eq:app4} that Theorem~\ref{thm:HW} and~\ref{thm:main} 
can only give meaningful upper bounds for $|\s^+|$ if
\begin{equation}\label{eq:app5}
b_1\leq |\s^+|=\sum_{i=1}^{n}x_i\leq b_m.
\end{equation}
Similarly to the proof of~\eqref{eq:fin1}, we can now use~\eqref{eq:app5} to show (although we omit 
the details here) that Theorem~\ref{thm:main} implies the inequality $(I_k)$ in Theorem~\ref{thm:HW}, i.e.
\begin{equation}\label{eq:fin2}
|\s^-|\leq \left\lfloor\s^+_\av\right\rfloor\leq \s^+_\av \leq U_{I_k}.
\end{equation}
Finally, it follows from~\eqref{eq:fin1} and~\eqref{eq:fin2} that the bounds in Theorem~\ref{thm:main} 
are in general sharper or equivalent to the bounds in Theorem~\ref{thm:HW}. 

\subsection*{Acknowledgements}
The author thanks Alfred Geroldinger for pointing to and providing background material related to 
zero-sum sequences and their applications to Factorization Theory. The author also thanks the reviewer
for making valuable comments.

\end{document}